%% file: bsnew.tex
\documentclass[a4paper,12pt]{article}
\usepackage{amsmath}
\usepackage{amssymb}
\usepackage{theorem}
\usepackage{color}
\usepackage{epsfig}
\usepackage{graphicx}

\theorembodyfont{\itshape}
\newtheorem{theorem}{Theorem}[section]

\newtheorem{corollary}[theorem]{Corollary}
\newtheorem{lemma}[theorem]{Lemma}
\theorembodyfont{\upshape}
\newtheorem{remark}[theorem]{Remark}

\newtheorem{definition}[theorem]{Definition}
\newtheorem{example}[theorem]{Example}

\makeatletter
\DeclareRobustCommand{\qed}{%
  \ifmmode 
  \else \leavevmode\unskip\penalty9999 \hbox{}\nobreak\hfill
  \fi
  \quad\hbox{\qedsymbol}}
\newcommand{\openbox}{\leavevmode
  \hbox to.77778em{%
  \hfil\vrule
  \vbox to.675em{\hrule width.6em\vfil\hrule}%
  \vrule\hfil}}
\newcommand{\qedsymbol}{\openbox}
\newenvironment{proof}[1][\proofname]{\par
  \normalfont
  \topsep6\p@\@plus6\p@ \trivlist
  \item[\hskip\labelsep\itshape
    #1\@addpunct{.}]\ignorespaces
}{%
  \qed\endtrivlist
}
\newcommand{\proofname}{Proof}
\makeatother

\newcommand{\C}{\mathbb{C}}
\newcommand{\R}{\mathbb{R}}
\newcommand{\Z}{\mathbb{Z}}

\renewcommand{\tilde}{\widetilde}
\renewcommand{\setminus}{\smallsetminus}

\newcommand{\Ind}{\mathrm{Ind}\,}
\newcommand{\fiber}{\mathrm{fiber}}
\newcommand{\Ker}{\mathrm{Ker}\,}

\title{Torus fibrations and localization of index I \\ 
{\large  --  Polarization and acyclic fibrations
}}
\author{Hajime Fujita, 
Mikio Furuta\thanks{Partially supported by JSPS 
Grant 19340015, 19204003 and 16340015.} 
 and Takahiko Yoshida\thanks{Partially supported by Advanced Graduate Program in Mathematical Sciences at Meiji University, JSPS Grant-in-Aid for Young Scientists (B) 20740029, and Fujyukai Foundation.}}
\date{}

\begin{document}

\maketitle

\begin{abstract}
We define a local Riemann-Roch number for
an open symplectic manifold when
a complete integrable system without Bohr-Sommerfeld
fiber is provided on its end. In particular when a structure of a singular Lagrangian fibration is given on a closed symplectic manifold, its Riemann-Roch number is described as the sum of the number of nonsingular Bohr-Sommerfeld fibers and a contribution of the singular fibers. A key step of the proof is formally explained as a version of Witten's deformation applied to a Hilbert bundle. 
\end{abstract}

\vspace{3mm}


\section{Introduction}

The purpose of this paper is to give a 
localization technique
for the index of spin$^c$ Dirac operator.
Our localization 
makes use of no group action, but a family of 
acyclic flat connections on tori.
A typical example is given by a symplectic manifold
with an integrable system.

For a complete integrable system on a 
closed symplectic
manifold, the Riemann-Roch number is 
sometimes equal to 
the number of Bohr-Sommerfeld fibers. 
For K\"ahler case D. Borthwick, T. Paul and A. Uribe
gave a relationship between
Bohr-Sommerfeld Lagrangians and
the kernels of twisted spin$^{c}$ Dirac operators
using  Fourier integral operators~\cite{BPU}.
A problem here would be how
to count singular Bohr-Sommerfeld fibers, and
how to count the contribution from other 
singular fibers if there are.
M. D. Hamilton and E. Miranda dealt with 
the singular Bohr-Sommerfeld fibers using 
J. \'Sniatycki's framework and found a discrepancy 
between quantizations via real and complex polarizations 
under  this framework~\cite{Hamilton, Hamilton-Miranda}.

When an integrable system is associated to a
Hamiltonian group action, 
the geometric quantization conjecture
of V.~Guillemin and S.~Sternberg is a localization property of
the Riemann-Roch character.
H. Duistermaat, V. Guillemin, E. Meinrenken 
and S. Wu gave a proof of the conjecture 
for torus actions using the geometric localization
provided by E. Lerman's symplectic cuts
\cite{DGMW} \cite{Lerman}.
Y.~Tian and W.~Zhang gave a proof using
a direct analytic localization 
with perturbation of the Dirac operator \cite{TianZhang}.

In this paper we give a localization property
of the Riemann-Roch numbers 
for a complete integrable system possibly with singular fibers.
We define multiplicities 
for Bohr-Sommerfeld fibers and singular fibers 
so that the total sum of the multiplicities 
for all the  fibers
is exactly equal to the Riemann-Roch number.
Our method is flexible and allows various generalizations.
In this paper, however, we explain only the simplest case.
In our subsequent paper we will explain 
some of the generalizations and, as an application, an approach to
the V.~Guillemin-S.~Sternberg conjecture for
Hamiltonian torus actions \cite{FFY,FFY3}.

Our idea is simple.
Let $M$ be a $2n$-dimensional closed symplectic manifold
with  a prequantizing line bundle $L$.
Suppose $X$ is an $n$-dimensional affine space, 
and $\pi:M \to X$ a completely integrable system.
The symplectic structure $\omega$ gives rise to
an element $[D\otimes L]$ of the K-homology group $K_0(M)$,
where $D$ is the Dolbeault operator for an almost complex
structure compatible with $\omega$.
The projection $\pi$ gives an element of $K_0(X)$
defined to be the pushforward $\pi_* [D\otimes L]$.
The Riemann-Roch number is 
calculated as the further pushforward of $\pi_*[D \otimes L]$
with respect to the map from $X$ to a point.

If we have a formulation of K-homology group
in terms of some geometric data, and if
a geometric representative of $\pi_*[D \otimes L]$
has a {\it localized  support}, then
the Riemann-Roch number can be calculated by
the data on the support.

In fact it is possible to realize this
idea rigorously if we use a formulation of K-homology
in terms of some notion of generalized vector
bundle with Clifford module bundle action,
which is developed in \cite{index 2}, \cite{Gomi}
\footnote{Strictly speaking what we actually need 
here is the notion of a 
{\it K-cohomology cocycle
with local coefficient} which is defined using some action of a 
Clifford algebra bundle.
For a manifold if the Clifford algebra
module is the one generated by its tangent space, 
the twisted K-cohomology group is identified with its
K-homology through a duality.}.  
In this short paper, however,  we explain only the
localization property 
without appealing to the framework of \cite{index 2}
at the expense of giving up 
identifying $\pi_*[D \otimes L]$
in terms of geometric data.

Some key points of our arguments are:
\begin{enumerate}
\item
The restriction of 
the Dolbeault operator, or the spin${}^c$ Dirac operator,
to a Lagrangian submanifold is equal to
its de Rham operator
at least on the level of a principal symbol.
\item
A flat line bundle over a torus is trivial if and only if
the corresponding twisted de Rham complex is not acyclic.
\item
The localization we use is given by an adiabatic limit.
\item
The Laplacian corresponding to the de Rham operator along
Lagrangian fibers plays the role of the potential term
of the spin${}^c$ Dirac operator on the base manifold.
\item
Since the Laplacian is a second-order elliptic differential
operator,
when it is strictly positive, it can absorb
the effect of first-order term.
\end{enumerate}
The last property makes our construction 
rather flexible.

The organization of this paper is as follows.
In Sections 2, 3 and 4, 
we give localization properties
for symplectic setting, topological setting,
and analytical setting respectively.
The latter is a generalization  of the former
for each stage.
In  Section 5 we prove the localization.
In Section 6 we give an example for $2$-dimensional
topological case.
Finally in Section 7 we give comments for
possible generalizations.

\section{Symplectic formulation}
Let $M$ be a $2n$-dimensional
closed symplectic manifold with 
a symplectic form $\omega$. 
Suppose $L$ is a complex Hermitian line bundle over $M$
 with Hermitian connection $\nabla$ satisfying
that the curvature of $\nabla$ is equal to $-2\pi\sqrt{-1} \omega$. 
$M$ has an almost complex structure compatible with $\omega$,
and we can define the Riemann-Roch number $RR(M, L)$
as the index of the Dolbeault operator with coefficients in $L$.

Let $X$ be an $n$-dimensional affine space,
and $\pi :M \to X$ a complete integrable system.
Then generic fibers of $\pi$ are empty or disjoint unions of 
finitely many $n$-dimensional tori
with canonical affine structures.

\begin{definition}
$x \in X$ is {\it $L$-acyclic} if $x$ is a regular value
of $\pi$ and $L$ does not have any non-vanishing parallel section
over the fiber at $x$.
\end{definition}



The main purpose of this paper is to show that
$RR(M,L)$ is localized at non $L$-acyclic points.
More precisely
\begin{theorem} \label{symplectic}
Suppose $X=U_{\infty} \cup (\cup_{i=1}^m U_i)$ 
is an open covering satisfying the following properties.
\begin{enumerate}
\item
$\{U_i\}_{i=1}^m$ are mutually disjoint.
\item
$U_\infty$ consists of $L$-acyclic points.
\end{enumerate}
Let $V_i=\pi^{-1}(U_i)$.
Then for each $i=1,\ldots,m$ there exists an
integer $RR(V_i, L)$, which depends only
on the data restricted on ${V_i}$, such that
$$
RR(M,L) = \sum_{i=1}^m RR(V_i,L).
$$
Here the integer $RR(V_i,L)$ is invariant under 
continuous deformations of the data. 
\end{theorem}

\begin{remark}
The theorem asserts that the Riemann-Roch number 
$RR(M,L)$ is localized at non-singular Bohr-Sommerfeld fibers 
and singular fibers. 
\end{remark}

\begin{remark}
While the Riemann-Roch number
$RR(M,L)$ for the closed symplectic manifold $M$
depends only on the symplectic structure $\omega$,
the localized Riemann-Roch number $RR(V_i, L)$ depends on
the restrictions of $\omega, \pi, \nabla$ as well,
though we omit this dependence in the notation
if there is no confusion.
\end{remark}

\begin{remark}
In fact Theorem~\ref{symplectic} holds for a 
Lagrangian fibration possibly with singular fibers. 
\end{remark}

In the next section we reduce the above localization 
theorem to 
a slightly more general localization 
(Theorem~\ref{topological})
formulated purely topologically. 

\section{Topological formulation}

In this section
let $M$ be a $2n$-dimensional
closed spin$^{c}$ manifold,\footnote{In this paper we take a convention of 
spin$^c$ structures which do not need any Riemannian metrics. See Appendix~\ref{spin^c-structures} 
for the convention.} and $E$ a complex Hermitian vector bundle over $M$. 
We define the Riemann-Roch number $RR(M, E)$
as the index of the spin${^c}$ Dirac operator with coefficients in 
$E$.\footnote{Precisely, in order to define the spin${^c}$ Dirac operator with coefficients in 
$E$ we need a unitary connection on $E$. But it is well-known that $RR(M, E)$ itself does not depend on the choice of connections. So we do not mention it here.}  
Let $V$ be an open subset of $M$.
\begin{definition}
{\it A real polarization} on $V$ is the data
$(U,\pi,\phi,J)$ satisfying the following properties.
\begin{enumerate}
\item
$U$ is an $n$-dimensional smooth manifold. 
\item
$\pi: V \to U$ is a fiber bundle whose fibers are 
disjoint unions of finitely many 
$n$-dimensional tori with affine structures.
\item
$\phi:\pi_*T_{\fiber}V \to TU$ is an isomorphism
between two real vector bundles, 
where $\pi_*T_{\fiber}V$ is the vector bundle on $U$ consisting of 
parallel sections of the tangent bundle of the fiber 
$\pi^{-1}(x)$ for each $x\in U$. 
\item
$J$ is an almost complex structure on $V$
which is a reduction of the given spin${}^c$-structure.
\item
The composition of $J:T_{\fiber} V \to TV$ and
$\pi_* :TV \to TU$ is equal to the map induced from $\phi$.
\end{enumerate} 
\end{definition}

Suppose that $V\subset M$ has a real polarization $(U,\pi, \phi,J)$  
such that the restriction $E|_V$ has 
a unitary connection $\nabla$ along fibers
for the bundle structure $\pi:V \to U$.
\begin{definition}
$(E, \nabla)$ is {\it acyclic} if 
the restriction $(E,\nabla)|_{\pi^{-1}(x)}$ is a flat vector bundle 
and the twisted de Rham cohomology group 
$H^* (\pi^{-1}(x), (E,\nabla)|_{\pi^{-1}(x)})$ is zero
for every $x \in U$.
\end{definition}

\begin{theorem} \label{topological}
Let $M$ be a closed spin${}^c$ manifold and
$E$ a complex Hermitian vector bundle over $M$.
Suppose $M=V_{\infty} \cup (\cup_{i=1}^m V_i)$ 
is an open covering satisfying the following properties.
\begin{enumerate}
\item
$\{V_i\}_{i=1}^m$ are mutually disjoint.
\item
$V_\infty$ has a real polarization $(U,\pi,\phi,J)$.
\item
$E|_{V_{\infty}}$ has a unitary connection $\nabla$ along fibers 
for the bundle structure $\pi: V_{\infty} \to U$.
\item
$(E|_{V_{\infty}}, \nabla)$ is acyclic.
\end{enumerate}
Then for each $i=1,\ldots,m$ there exists an
integer $RR(V_i, E)$, which depends only
on the data 
{restricted on $V_i$}, such that
$$
RR(M,E) = \sum_{i=1}^m RR(V_i,E).
$$
Here the integer $RR(V_i,E)$ is invariant under 
continuous deformations of the data. 
\end{theorem}

\noindent
\begin{proof}
[\bf Proof of Theorem~\ref{symplectic} assuming
Theorem~\ref{topological}.]
Note that the symplectic structure gives an isomorphism
$T^*U_{\infty} \cong \pi_*T_{\fiber} V_\infty$.
Fix a Riemannian metric on $TU_\infty$ so that
we have an isomorphism $TU_\infty \cong T^*U_\infty$.
Define $\phi$ by using these two.
By fixing a Lagrangian splitting $TV_{\infty}\cong T_{\fiber}V_{\infty}\oplus\pi^*TU_\infty$, 
$\phi$ induces an almost complex structure, 
which determines the spin$^c$ structure. 
We take $E$ to be $L$. Then the rest would be obvious.
\end{proof}

In the next section we further reduce
Theorem~\ref{topological} to 
more general localization for some Dirac-type
operator 
(Theorem~\ref{analytical}).

\section{Analytical formulation}

In this section
let $M$ be a 
closed Riemannian manifold.
We denote by $Cl(TM)$ the Clifford algebra bundle
over $M$ generated by $TM$.
Let
$W=W^0 \oplus W^1$ be
a ${\Z}/2\Z$-graded 
complex Hermitian vector bundle over $M$
with a structure of $Cl(TM)$-module
such that for each vector $v$ in $TM$,
the action of $v$ is skew-Hermitian and 
of degree-one.

We use the following definition of {\it Dirac-type operator}.
\begin{definition}
A first order differential operator $D:\Gamma(W) \to \Gamma(W)$
is {\it a Dirac-type operator on $W$} if $D$ is 
a degree-one formally self-adjoint
operator with smooth coefficient whose symbol is given
by the Clifford action on $W$.
\end{definition}

Dirac-type operators on $W$ are not unique, but their indices
are equal. 
\begin{definition}
We denote by $\Ind(M,W)$ 
the index of a Dirac-type operator on $W$.
\end{definition}

Let $V$ be an open subset of $M$.
\begin{definition}
{\it A generalized real polarization} 
on $V$ is the data 
$(U, \pi,D_{\fiber})$ satisfying the following properties.
\begin{enumerate}
\item
$U$ is a Riemannian manifold.
\item
$\pi: V \to U$ is a fiber bundle with
fiber a closed manifold.
\item
Let $TV = T_{\fiber}V \oplus T_{\fiber}^{\perp}V$ be
the orthogonal decomposition with respect to the Riemannian metric on $V$. Then
the projection gives an isometric isomorphism
$T_{\fiber}^{\perp}V \cong \pi^*TU$.
\item
$D_{\fiber}:\Gamma(W|_V) \to \Gamma(W|_V)$
is a family of Dirac-type operators along fibers
anti-commuting with the Clifford action of $TU$ in the
following sense.
\begin{enumerate}
\item
$D_{\fiber}$ is an order-one, formally self-adjoint differential 
operator of degree-one.
\item
$D_{\fiber}$ contains only the derivatives along fibers, i.e,
$D_{\fiber}$ commutes with multiplication of the pullback
of smooth functions on $U$.
\item
The principal symbol of $D_{\fiber}$ is given by
the Clifford action of $T_{\fiber} V$.
\item
The Clifford action of $TU$ on $W|_V$ 
anti-commutes with $D_{\fiber}$. 
Here the Clifford action of $TU$ on $W|_V$ is defined
through the horizontal lift 
$\pi^* TU \cong T_{\fiber}^{\perp}V \subset TV$, where
the first isomorphism is the one given in $3$ above.
\end{enumerate}
\end{enumerate} 
\end{definition}

\begin{definition}
A generalized real polarization $(U, \pi, D_{\fiber})$ on $V$ 
is {\it acyclic} if 
for each $x \in U$, the restriction of
$D_{\fiber}$ to 
$\Gamma(W|_{\pi^{-1}(x)})$ has  zero kernel.
\end{definition}

\begin{theorem} \label{analytical}
Let $M$ be a closed Riemannian manifold and
$W=W^0 \oplus W^1$ a $Cl(TM)$-module bundle as above. 
Suppose $M=V_{\infty} \cup (\cup_{i=1}^m V_i)$ 
is an open covering satisfying the following properties.
\begin{enumerate}
\item
$\{V_i\}_{i=1}^m$ are mutually disjoint.
\item
$V_\infty$ has an acyclic generalized 
real polarization $(U, \pi, D_{\fiber})$.
\end{enumerate}
Then for each $i=1,\ldots,m$ there exists an
integer $\Ind(V_i, W)$, which depends only
on the data 
{restricted on $V_i$}, such that
$$
\Ind(M,W) = \sum_{i=1}^m \Ind(V_i,W).
$$
Here the integer $\Ind(V_i,W)$ is invariant under 
continuous deformations of the data. 
\end{theorem}

\noindent
The proof of Theorem~\ref{topological} follows from 
the next obvious lemma.

\begin{lemma}
Let $T$ be an $n$-dimensional torus with an affine structure.
Let ${\mathfrak X}$ be the $n$-dimensional vector space of parallel vector fields.
For any Euclidean metric on ${\mathfrak X}$, 
$T$ has an induced flat Riemannian metric.
Then each element of the dual space 
${\mathfrak X}^*$ gives a harmonic $1$-form.
\end{lemma}

\begin{proof}
[\bf Proof of Theorem~\ref{topological} assuming
Theorem~\ref{analytical}.] 
Fix a Riemannian metric on $U$. 
Combining it with the flat metric associated with 
the affine structures on fibers of $V_{\infty}$ via 
the almost complex structure $J$, 
we define a Riemannian metric on $V_{\infty}$ 
and extend it to $M$. 
Take $W$ to be 
the tensor product of the spinor bundle 
of the spin$^c$ manifold $M$ and $E$. 
Then we define the family of Dirac-type operators along fibers 
acting on $\Gamma(W|_{V_{\infty}})$ 
by $D_{\fiber}:=d_{E,\fiber}+d_{E,\fiber}^*$, where 
$d_{E,\fiber}$ is the exterior derivative on fibers 
twisted by the unitary connection $\nabla$ on $E$ and 
$d_{E,\fiber}^*$ is its formal adjoint. 
The above lemma implies the anti-commutativity 
between $D_{\fiber}$ and the Clifford action of $TU$. 
The acyclic condition for $(E,\nabla)$ implies the acyclicity 
for $(U,\pi,D_{\fiber})$. 
\end{proof}
\section{Local index}

Let $M$  be a Riemannian manifold and
$W=W^0 \oplus W^1$ a $Cl(TM)$-module bundle.
Suppose $V$ is an open subset of $M$ with
an acyclic generalized real polarization
$(U, \pi,D_{\fiber})$ such that
$M \setminus V$ is compact.
We will define the {\it local index} $\Ind(M,V,W)$ 
(or $\Ind(V,W)$ for short) 
and show deformation invariance and an excision property.
The local index
depends on the acyclic generalized real polarization on $V$ though
we omit it in the notation for simplicity.
\begin{remark}
When $\pi:V \to U$ is a diffeomorphism,
$D_{\fiber}$ is a degree-one
self-adjoint homomorphism on $W|_V$ anti-commuting
with the Clifford multiplication of $TV$.
In this special case the definition of $\Ind(V,W)$
is already given in \cite[Chapter 3]{index 1}.
(See Definition 3.14 for the setting,
Definition 3.21 for the definition in the case of
cylindrical end,
Theorem 3.20 for deformation invariance,
Theorem 3.29 for the excision property, and
Section 3.3 for the definition for general case.)
We will generalize the argument there.
\end{remark}

\subsection{Vanishing lemmas}
We will show the following lemma later.
\begin{lemma} \label{vanishing for closed case}
Suppose $M$ is closed and $M=V$ has an acyclic generalized real polarization. 
Take any Dirac-type operator $D$ on $W$ and
write $D=\tilde{D}_U + D_{\fiber}$ for some operator
$\tilde{D}_U$. 
For a real number $t$, put $D_t=\tilde{D}_U + t D_{\fiber}$.
Then for any large $t>0$, $\Ker D_t=0$.
\end{lemma}

We also need a slightly generalized version,
which is shown later.

\begin{lemma}\label{vanishing for cylindrical case}
Suppose $M=V$ and $M$ has a cylindrical end with
translationally invariant acyclic generalized real polarization on it.
Take any Dirac-type operator $D$ on $W$ with
translationally invariance on the end,
and
write $D=\tilde{D}_U + D_{\fiber}$ for some operator
$\tilde{D}_U$. 
For a real number $t$, put $D_t=\tilde{D}_U + t D_{\fiber}$.
Then for any large $t>0$, 
$\Ker D_t \cap \{ {\mbox{\rm $L^2$-sections}} \} =0$.
\end{lemma}
Admitting these lemmas
we first give the definition and properties of 
the local index.

\subsection{Cylindrical end}

We first give the definition for the
special case that $M$ has a cylindrical end and
every data is translationally invariant on the end.

\begin{lemma}\label{invariance for cylindrical case}
Suppose $M$ has a cylindrical end $V=N \times (0,\infty)$
with translationally invariant acyclic generalized real polarization
on it. Let $\rho$ be a non-negative smooth cut-off function on $M$
satisfying $\rho =1$ on $N \times [1,\infty)$ and
$\rho=0$ on $M \setminus V$. 
For $t>1$, put $\rho_t:=1+t\rho$. 
Take any Dirac-type operator $D$ on $W$ with
translationally invariance on the end,
and
write $D=\tilde{D}_U + D_{\fiber}$ for some operator
$\tilde{D}_U$ on the end.
Put $D_t=\tilde{D}_U +\rho_t \, D_{\fiber}$
on the end and $D_t=D$ on $M \setminus V$.
Then for any large $t>0$, 
$\Ker D_t \cap \{ {\mbox{\rm $L^2$-sections}} \}$
is finite dimensional.
Moreover its super-dimension is independent of
large $t$ and any other continuous deformations
of data.
\end{lemma}
\begin{proof}
The restriction of $D_t$ to $N \times [1,\infty)$
is of the form $\alpha (\partial_r + D_{N,t})$ where $\alpha$ is
the Clifford multiplication of $\partial_r$
and $D_{N,t}$ is a formally self-adjoint operator
on $N$ depending on the parameter $t$.
We show that $\Ker D_{N,t}=0$ for large $t$.
It is technically convenient to introduce
a Dirac-type operator $D_{N \times S^1,t}$ on
$N \times S^1$ as follows:
$D_{N \times S^1,t}$ is written 
as the same expression $\alpha (\partial_r + D_{N,t})$
where we use the identification $S^1=\R/\Z$.
Since $N \times S^1$ is a closed manifold, 
we can apply  
Lemma~\ref{vanishing for closed case} to obtain
$\Ker D_{N \times S^1,t}=0$ for large $t$,
which implies our claim $\Ker D_{N,t}=0$ for large $t$.


When $D_{N,t}$ does not have zero as an eigenvalue, 
any $L^2$-solution $f$ for
the equation $D_t f=0$ on $M$ is exponentially decreasing
on the end. Then it is well-known that
the space of $L^2$-solutions is finite dimensional,
and its super-dimension is deformation invariant
as far as $\Ker D_{N,t}=0$. 
\end{proof}
The super-dimension of the space of $L^2$-solutions
is equal to the index of $D_t |_{M \setminus V}$
for the Atiyah-Patodi-Singer boundary
condition. We use this index as the definition of
our local index for the case of cylindrical end. 
\begin{definition}
\label{definition for cylindrical case}
Under the assumption of 
Lemma~\ref{invariance for cylindrical case},
$\Ind(M,V,W)$ is defined to be
the super-dimension of
$\Ker D_t \cap \{ {\mbox{\rm $L^2$-sections}} \}$.
\end{definition}
The following sum formula follows from a standard argument.
\begin{lemma}\label{sum formula}
Suppose $(M,V=N \times (0,\infty),W)$
and $(M', V' = N' \times (0,\infty),W')$
satisfy the assumption of 
Lemma~\ref{invariance for cylindrical case}.
Let $N_0$ and $N'_0$ be a connected component of
$N$ and $N'$ respectively.
Suppose $N_0$ is isometric to $N'_0$
via $\phi:N_0 \to N'_0$,
and for some $R>0$ 
the map $\phi:N_0 \times (0, R) \to N'_0 \times (0,R)$
given by $(x, r) \mapsto (\phi (x), R-r)$ can be
lifted to the isomorphism between the acyclic generalized real polarizations
on them. 
Then we can glue $M\setminus (N_0 \times [R,\infty))$
and $M' \setminus (N'_0 \times [R,\infty))$
to obtain a new manifold $\hat{M}$
with cylindrical end $\hat{V}=\hat{N} \times (0,\infty)$ for
$\hat{N}=(N\setminus N_0)\cup (N' \setminus N'_0)$, 
and we also have a Clifford module bundle $\hat W$ 
obtained by gluing $W$ and $W'$ on 
$N_0 \times (0, R) \cong N'_0 \times (0,R)$. 
Then we have 
$$
\Ind(\hat{M},\hat{V},\hat W)=\Ind(M,V,W)+\Ind(M',V',W'). 
$$
\end{lemma}
\begin{proof}
A proof is given by the APS formula of the indices. An alternative direct argument is to apply the excision property~\cite[Theorem~5.40]{index 1}. 
\end{proof}

\subsection{General case}

Now we would like to define the local index for
general case.

Let $V$ be an open subset of $M$ with
an acyclic generalized real polarization $(U,W_U, \pi,D_{\fiber})$
such that
$M \setminus V$ is compact. 
We can take a codimension-one closed submanifold 
$N_U$ of $U$ such that $N=\pi^{-1}(N_U)$
divides $M$ into compact part and non-compact part:
For instance, let $f: M \to [0,\infty)$ be
the distance from $M\setminus V$ and define
$g: U\to [0,\infty)$ to be the maximal value
of $f$ on the fiber of $\pi$.
Take a small real number
$r>0$ such that $f^{-1}([0,r])$ is a compact subset
of $M$. Let $h:U \to [0,\infty)$ be a
smooth functions on $U$ satisfying $|h(x)-g(x)| <r/2$ for
$x \in U$. Take a regular value $r_0$ of $h$ satisfying
$0<  r_0 <r/2$. Then $N_U=h^{-1}(r_0)$ satisfies
the required property.

Let $K$ be the compact part of $M \setminus N$.
Note that a neighborhood $V_N$ of $N$ in $V$
is diffeomorphic to
$N \times (-\epsilon, \epsilon)$.
Then we can construct the Riemannian metric 
and the translationally invariant acyclic 
generalized real polarization 
on $(M',V')$, where 
$M':=K\cup (N\times [0,\infty))$ and 
$V':=(K\cap V)\cup (N\times [0,\infty))$. 
For instance, let $\phi:M' \to M$ be
a smooth map which is the identity
on the complement of $N\times (-\epsilon,\infty)$ and
is given by 
$(x,r) \mapsto (x, \beta(r))$ on
$N\times (-\epsilon,\infty)$ 
for a smooth function $\beta$ satisfying
$\beta(r)= r$ for 
$-\epsilon <r<-(2/3)\epsilon$, and 
$\beta(r)=0$ for 
$r \geq(1/3)\epsilon$.
Define a bundle endmorphism 
$\tilde{\phi}: TM' \to TM$ covering $\phi$ 
as follows. On the complement of $N\times (-\epsilon,\infty)$, 
$\tilde{\phi}$ is the identity. 
On $T(N\times (-\epsilon,\infty))=TN \times T(-\epsilon,\infty)$, 
define $\tilde{\phi}$ by
$((x,r),(u,v)) \mapsto (\phi(x,r),(u,v))$,
where $x \in N$, $r \in (-\epsilon,\infty)$, 
$u \in T_xN$ and 
$v \in T_r(-\epsilon,\infty)={\R}$. 
The required deformed Riemannian metric
is defined to be the pullback of the original
Riemannian metric by $\tilde{\phi}$ 
as a section of the symmetric tensor product of $T^*M$.
The required deformed Clifford module bundle $W'$ is
defined to be the pullback $\phi^* W$. 
Note that we can also construct a Riemannian manifold $U'$ 
with a cylindrical end $N_U\times (0,\infty)$ and 
a fiber bundle $\pi':V'\to U'$. 
We define a Dirac-type operator $D_{\fiber}'$ acting on $\Gamma(W'|_{V'})$ 
by $\phi^*D_{\fiber}$. 
Then $(U',\pi',D_{\fiber}')$ is a 
translationally invariant acyclic generalized real polarization 
on $(M', V',W')$. 
For this structure the local index is defined 
by Definition~\ref{definition for cylindrical case}.

\begin{definition}\label{definition for general case}
We define $\Ind(M,V,W)$ to be the local index
for $(M', V',W')$ with 
the translationally invariant acyclic generalized real polarization 
$(U',\pi',D_{\fiber}')$. 
\end{definition}
We have to show the local index is well-defined
for the various choice of our construction.
\begin{lemma} \label{well-definedness of Ind}
Suppose we take two codimension-one closed submanifolds
$N_U$ and $N_U'$ in $U$ so that $M$ is divided
in two ways.
Then the local indices defined by these
data coincide. 
\end{lemma}
\begin{proof}
Let $K$ and $K'$ be the compact parts of $M$
divided by $N=\pi^{-1}(N_U)$ and $N'=\pi^{-1}(N_U')$
respectively.
Then we can take another $N_U''$ so that
the corresponding compact part $K''$ is contained
in the intersection of the interiors of $K$ and $K'$.
It suffices to show that the local indices
coincide for $K$ and $K''$.
Deform the structures on neighborhoods of $K$ and $K''$
simultaneously to make the structures translationally
invariant near $K$ and $K''$ respectively.
Let $M_0,M_0'$ and $\hat{M}_0$ be
the following manifolds with cylindrical ends.
\begin{eqnarray*}
M_0&=&K'' \cup (N'' \times [0,\infty)) \\
M_0'&=&(N'' \times (-\infty, 0]) \cup (K\setminus K'')
 \cup (N \times [0,\infty)) \\
\hat{M_0}&=&K \cup (N \times [0,\infty))
\end{eqnarray*}
On the cylindrical ends we have 
translationally invariant Clifford modules
$W_0,W_0'$ and $\hat{W}_0$, and
translationally invariant acyclic generalized real polarizations.
On $M_0'$ the acyclic generalized real polarization is given
globally.
The sum formula of Lemma~\ref{sum formula} implies
$\Ind(M_0,W_0)+\Ind(M_0',W_0')=\Ind(\hat{M}_0,\hat{W}_0)$.
The vanishing of Lemma~\ref{vanishing for cylindrical case}
implies $\Ind(M_0',W_0')=0$.
Therefore we have $\Ind(M_0,W_0)=\Ind(\hat{M}_0,\hat{W}_0)$.
This is the required equality.
\end{proof}
\subsection{Excision}
The well-definedness shown in
Lemma~\ref{well-definedness of Ind} is
the key point for the following formulation
of excision property.
\begin{theorem}(Excision property)
\label{exicion}
Let $W$ be a $\Z_2$-graded Clifford module bundle over
$Cl(TM)$.
Let $V$ be an open subset of $M$ with
an acyclic generalized real polarization $(U,W_U,\pi,D_{\fiber})$
such that $M \setminus V$ is compact.
Suppose $U'$ is an open subset of $U$
such that 
$M':=V' \cup (M \setminus V)$ 
is an open neighborhood of $M \setminus V$, 
where we put $V':=\pi^{-1}(U')$. 
Note that
$V'$ has the restricted acyclic generalized real polarization.
Then we have
$$
\Ind(M,V,W)=\Ind(M',V',W|_{M'}). 
$$
\end{theorem}
\begin{proof}
Take a codimension-one submanifold $N_{U'}$ in $U'$
to define $\Ind(M',V',W|_{M'})$. Then $N_{U'}$ can be
used to define $\Ind(M,V,W)$.
\end{proof}

\begin{proof}
[Proof of Theorem~\ref{analytical}]
We first note that when $M$ is a closed manifold 
the local index $\Ind(M,V,W)$ defined 
by Definition~\ref{definition for general case} is 
equal to the usual index $\Ind(M,W)$ of a 
Dirac-type operator. 
Under the assumption of Theorem~\ref{analytical}
from the excision property we have
$\Ind(M, W)=\Ind(\cup_{i=1}^m V_i, W|_{\cup_{i=1}^m V_i})$,
which implies the theorem.
\end{proof}




\subsection{Proof of vanishing lemmas}

In this subsection we show
the vanishing lemmas Lemma~\ref{vanishing for closed case} 
and Lemma~\ref{vanishing for cylindrical case}. 
Suppose $V$ is an open subset of $M$
with an acyclic generalized real polarization
$(U, \pi, D_{\fiber})$. 
Take any order-one formally self-adjoint 
differential operator $\tilde{D}$ over $W|_V$ with degree one
whose principal symbol is given by the composition of 
the projection
$\pi_*: TV \to TU$ and the Clifford action of $TU$ on $W|_V$.
Then $\tilde{D} + D_{\fiber}$ is a Dirac-type operator
on $W|_V$.

\begin{lemma}
The anticommutator $D_{\fiber}\tilde{D}+\tilde{D}D_{\fiber}$
is an order-one differential operator on $W|_V$ which
contains only the derivatives along fibers, i.e,
it commutes with the multiplication of the pullback
of smooth functions on $U$.
\end{lemma}

\begin{proof}
Recall that,
the principal symbol of $\tilde{D}$ anti-commutes 
not only with the symbol of $D_{\fiber}$, but also 
with the whole operator $D_{\fiber}$.
The claim follows from this property.
It is straightforward to check it using local description.
Instead of giving the detail of the local calculation, however,
we here give an alternative formal explanation for the above lemma.

For $x \in U$ let ${\mathcal W}_x$ be
the space of sections of the restriction of $W$ to the fiber $\pi^{-1}(x)$.
Then ${\mathcal W}=\coprod {\mathcal W}_x$ is formally
an infinite dimensional vector bundle over $U$. 
We can regard $D_{\fiber}$ as an endmorphism on ${\mathcal W}$.
Then $D_{\fiber}$ is a order-zero differential operator
on ${\mathcal W}$ whose principal symbol is equal to 
$D_{\fiber}$ itself.
Then, as a differential operator on ${\mathcal W}$,
$D_{\fiber}\tilde{D}+\tilde{D}D_{\fiber}$ is
an (at most) order-one operator whose principal symbol
is given by the anticommutator between 
the Clifford action by $TU$
and $D_{\fiber}$. This principal symbol vanishes, which implies
that the anticommutator is order-zero as a differential operator
on ${\mathcal W}$, i.e., it does not contain derivatives of
$U$-direction.
\end{proof}

\begin{proof}
[Proof of Lemma~\ref{vanishing for closed case}]
Let $f$ be a section of $W$.
On each fiber of $\pi$ at $x\in U$, 
the second order
elliptic operator
$D_{\fiber}^2$ is strictly positive.
Since $D_{\fiber}\tilde{D}+\tilde{D}D_{\fiber}$
gives a first order operator on the fiber,
a priori estimate implies the estimate
$$
\left|
\int_{\pi^{-1}(x)} 
((D_{\fiber}\tilde{D}+\tilde{D}D_{\fiber})f,f)
\right|
\leq C \int_{\pi^{-1}(x)} (D_{\fiber}^2 f,f)
$$
for some positive constant $C$.
Since $M$ is compact we can take $C$ uniformly.
Therefore we have
\begin{eqnarray*}
\int_{M} ((\tilde{D} + t D_{\fiber})^2 f,f)
&=&
\int(\tilde{D}^2 f,f) +t^2 \int(D_{\fiber}^2 f,f) \\
&&
+t \int((D_{\fiber}\tilde{D}+\tilde{D}D_{\fiber})f,f)
 \\
&\geq&
\int(\tilde{D}^2 f,f) +
(t^2-Ct) \int(D_{\fiber}^2 f,f) \\
&=&
\int_M |\tilde{D} f|^2
+ (t^2 -Ct) \int_M |D_{\fiber} f|^2.
\end{eqnarray*}
In particular if $t>C$ and
$(\tilde{D} + t D_{\fiber})f=0$, then
$D_{\fiber}f$ is zero, which implies $f$
itself is zero.
\end{proof}

\begin{proof}
[Proof of Lemma~\ref{vanishing for cylindrical case}]
The proof is almost identical to the
above one for the compact case.
What we still need is to guarantee the validity
of partial integration.
This validity follows from the fact that
when $(\tilde{D} + t D_{\fiber})f=0$ and $f$
is in $L^2$, then
$f$ and any derivative of $f$ decay
exponentially on the cylindrical end of $V$.
\end{proof}
Note that the exponential decay in the above
proof relies on
the vanishing lemma for compact case, so
we had to separate the proof.

\section{2-dimensional case}

Let $\Sigma$ be a compact oriented surface with non-empty boundary
$\partial \Sigma$.
Let $L$ be a complex line bundle over $\Sigma$.
Suppose a $U(1)$-connection is given on the restriction of
$L$ to  $\partial \Sigma$.
When the connection is non-trivial on every boundary component,
we can define the local Riemann-Roch number for the data
as follows.
Fix a product structure $(-\epsilon, 0] \times \partial \Sigma$
on an open neighborhood of the boundary.
Then on the collar neighborhood of each connected component of $\partial \Sigma$ the projection onto 
$(-\epsilon, 0]$ is a circle bundle.
Extend the connection smoothly on the neighborhood of the boundary
so that we have a flat non-trivial connection on each fiber
of the circle bundle.
Let $V=V_1$ be the interior of $\Sigma$, and $V_{\infty}$ be the intersection
of the open collar neighborhood and $V$.
Then we have the local Riemann-Roch number $RR(V,L)$ (see Theorem~\ref{topological}). 
The deformation invariance of the local Riemann-Roch number implies
that it depends only on the initially given data.
We often write 
$$
[\Sigma]=RR(V,L).
$$
In this section we calculate $[\Sigma]$ explicitly for several examples
(Theorem~\ref{local RRs in 2 dim}). We also show that the local Riemann-Roch number
for a non-singular Bohr-Sommerfeld fiber in symplectic case is equal to one 
(Theorem~\ref{local RR of BS}). 
Let us first recall our convention of orientation for
boundary. 
We use the convention for which Stokes theorem holds with  positive sign.
In other words:
Suppose $\hat{X}$ is an oriented manifold, and
$f$ is a smooth real function on $\hat{X}$
with $0$ a regular value.
The orientations of $X=f^{-1}((-\infty,0])$ and $Y=f^{-1}(0)$
are related to each other as follows.
If $\omega_X$ and $\omega_Y$ are
non-vanishing top-degree forms on
$X$ and $Y$ compatible with
their orientations, then we have 
$
\omega_X |_Y= \rho\,df \wedge \omega_Y
$
for some positive smooth function $\rho$ on $Y$.

\subsection{Type of singularities}

\subsubsection{BS type singularities}


For a positive number $\epsilon$ let $A_\epsilon$ 
be the annulus
$[-\epsilon,\epsilon] \times S^1$
with the orientation given by 
$dx \wedge d\theta / 2 \pi$,where
$x$ is the coordinate of $[-\epsilon,\epsilon]$.
The projection map $(x,\theta) \mapsto x$ gives a circle bundle
structure of $A_\epsilon$.
Let $L$ be a complex line bundle over $A_\epsilon$, and
$\nabla$ a $U(1)$-connection on $L$.
Let $e^{\sqrt{-1} h(x)}$ be the
holonomy along the circle of $\{ x \} \times S^1$
for the orientation given by $d \theta$.
Suppose $h(x)$ is continuous and  $h(0)=0$.

\begin{definition}[positive/negative BS]
When $h(x)>0$ for  $x>0$ and
$h(x)<0$ for $x<0$, we call the fiber at $0$
a {\it  positive BS type}.
When $h(x)<0$ for  $x>0$ and
$h(x)>0$ for $x<0$, we call the fiber at $0$
a {\it  negative BS type}. See Figure~\ref{fig1}. 
\end{definition}
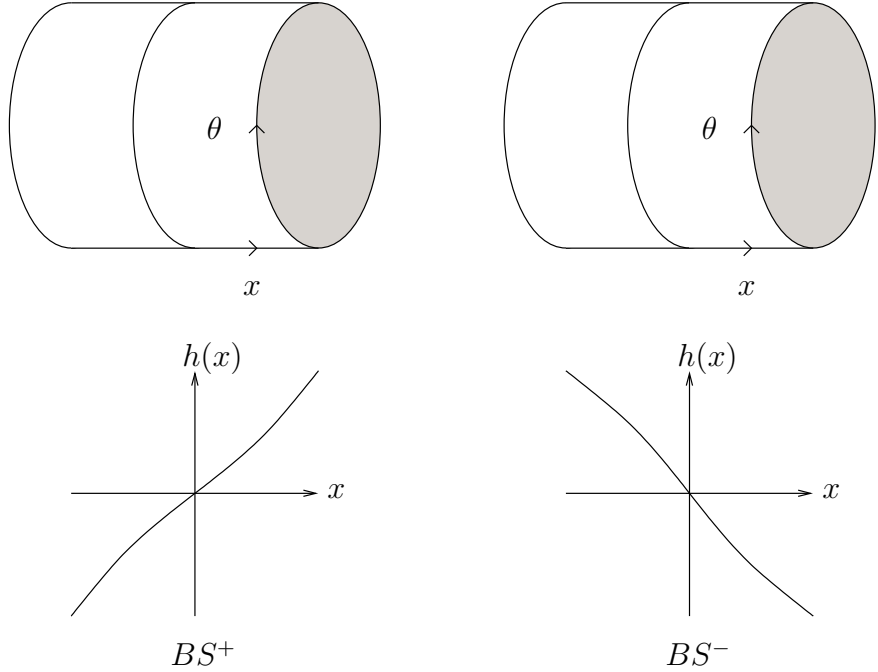
\begin{figure}[hbtp]
\begin{center}
\input{BS.tex}
\caption{positive/negative BS}
\label{fig1}
\end{center}
\end{figure}

\subsubsection{Disk type singularities}

Let $D^2$ be an oriented disk.
Choose a polar coordinate
$r$ and $\theta$ so that $D^2$ becomes a unit disk and
the orientation of $D^2$ is compatible with
$dr \wedge d\theta$ outside the origin.
In particular  the orientation of
the boundary $\partial D^2$ is compatible with
$d\theta$.
The projection map $(r,\theta) \mapsto r$ gives a circle bundle
structure
on neighborhood of the boundary of $D^2$.

Let $L$ be a complex line bundle over $D^2$, and
$\nabla$ a $U(1)$-connection on $L$.
Let $e^{\sqrt{-1} h(r)}$ be the
holonomy along the circle of radius $r$ centered in the origin.
The orientation of the circle is defined as the boundary
of the disk of radius $r$ centered in the origin.
We can take $h(r)$ continuous with limit value $\lim_{r \to 0}h(r)=0$.

\begin{definition}[positive/negative disk]
When $h(r)>0$ for small $r>0$ we call the singularity 
positive-disk type.
When $h(r)<0$, then we call it negative-disk type. See Figure~\ref{fig2}. 
\end{definition}
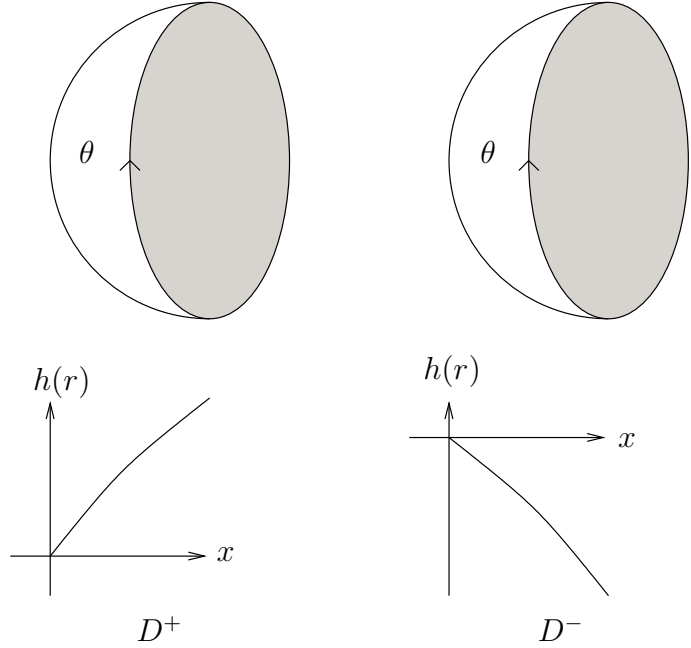
\begin{figure}[hbtp]
\begin{center}
\input{Disc.tex}
\caption{positive/negative disk}
\label{fig2}
\end{center}
\end{figure}
\subsubsection{Pants type singularities}

Let $\Sigma$ be a genus $0$ oriented surface with three holes, i.e.,
$\Sigma$ is a pair of pants.
For each boundary component its collar is diffeomorphic to
the product of a circle and an interval. The projection
onto the interval gives a circle bundle structure near the boundary.

Let $L$ be a complex line bundle over $\Sigma$, and $\nabla$ a $U(1)$-connection on $L$. Suppose $\nabla$ is flat and its holonomy along 
each component of the boundary of $\Sigma$ is non-trivial.
Let
$e^{\sqrt{-1}h_k}$ $(k=1,2,3)$ be the three holonomies along
the three components of the boundary of $\Sigma$.
The orientations of the boundary components are defined
as the boundary of the oriented manifold $\Sigma$.
Then the product of these three holonomies is equal to $1$.
We can take $h_k$ satisfying $0 < h_k <2\pi$.

From our assumption, there is no Bohr-Sommerfeld fiber
in the collar neighborhood of the boundary.

\begin{definition}[small/large pants]
When $h_1+h_2+h_3=2 \pi$, we call $\Sigma$ a small pants.
When $h_1+h_2+h_3=4 \pi$, we call $\Sigma$ a large pants.
\end{definition}

\subsection{Examples}
\begin{example}[a torus over a circle with degree $n$ line bundle]
\label{A torus over a circle with degree $n$ line bundle}
Let $x$ and $y$ the coordinate of
${\R}^2$.
Let $\hat\nabla$ be the $U(1)$-connection on the trivial complex line
bundle over ${\R}^2$ with connection form $-\sqrt{-1} x dy$.
\begin{enumerate}
\item
The curvature $F_{\hat \nabla}$ of $\hat\nabla$ is $-\sqrt{-1} dx\wedge dy$.
In particular 
$$
\frac{\sqrt{-1}}{2\pi} F_{\hat\nabla} = \frac{1}{2\pi} dx \wedge dy
$$
\item
The holonomy along the straight line from the point $(x,0)$ to $(x,y)$
is equal to $\exp (\sqrt{-1} xy)$.
\item
The connection is invariant under the action of $(m,2\pi n)\in
{\Z} \oplus 2 \pi {\Z}$
given by
$$
(x,y,u) \mapsto (x+m, y+2 \pi n, e^{\sqrt{-1}my}u ).
$$
\end{enumerate}
For a positive integer $N$
let $T^2_N$ be the quotient of ${\R}^2$
divided by the subgroup $N{\Z} \oplus  2\pi {\Z}$,
on which we have the quotient complex line bundle $L_N$
with the quotient connection $\nabla_N$.
Then $(\sqrt{-1}/2 \pi) F_{\nabla_N}$ gives a symplectic structure
on $T^2_N$. The projection onto the first factor
$T^2_N \to {\R}/N\Z$ is a
Lagrangian fibration. 
The holonomy along the fiber at $x \bmod N$ is
$e^{2\pi \sqrt{-1} x}$.
The Bohr-Sommerfeld fibers
are the fibers at $x \in  {\Z}/ N\Z$, and
all of them are positive BS singuralities.
The degree of $L_N$ is equal to $N$ and
the Riemann-Roch number is $N$ from the Riemann-Roch theorem and
it is equal to the number of positive BS singularities.
\end{example}

\begin{example}[a sphere with degree-zero line bundle]
\label{Sphere with trivial line bundle}
Let $D^+$ and $D^-$ be disks, and $L^+$ and $L^-$ complex line bundles over $D^+$ and $D^-$ with $U(1)$-connections
such that the connections are flat near the boundaries, respectively.
Suppose $D^+$ and $D^-$  are a positive disk and a negative disk 
respectively, and
the holonomies along the boundaries are
$e^{\sqrt{-1}\epsilon}$ and $e^{-\sqrt{-1} \epsilon}$
for a small positive number $\epsilon$.
Patch $(D^+,L^+)$ and $(D^-,L^-)$ together to obtain
a complex line bundle over an oriented sphere with a $U(1)$-connection.
Since the degree of the $U(1)$-bundle is zero,
the Riemann-Roch number is $1$.
\end{example}

\begin{example}[a surface with a pants 
decomposition with a flat line bundle]
\label{Surface with a pants decomposition with flat line bundle}
Let $\epsilon$ be a small positive number.
Let $P^S$ be a pair of pants with a flat $U(1)$-bundle whose
holonomies along the boundary components are
$e^{\sqrt{-1}(\pi- \epsilon)}$, $e^{\sqrt{-1}(\pi- \epsilon)}$, and
$e^{\sqrt{-1} 2\epsilon}$. 
Let $P^L$ be a pair of pants with a flat $U(1)$-connection whose
holonomies along the boundary components are
$e^{\sqrt{-1} (\pi+\epsilon)}$, $e^{\sqrt{-1}(\pi+ \epsilon)}$, and
$e^{- \sqrt{-1} 2\epsilon}$. 
Then $P^S$ is a small pants and $P^L$ is a large pants.
For an integer $g \geq 2$,
take $(g-1)$ copies of $P^S$ and
$(g-1)$-copies of $P^L$, and patch them together
to obtain a flat connection on a closed oriented surface
with genus $g$.
The Riemann-Roch number is $1-g$.
\end{example}

\subsection{Local Riemann-Roch numbers}

Let $[BS^+]$ and $[BS^-]$
be the contribution of a positive and negative BS respectively.
Let $[D^+]$ and $[D^-]$be the contribution of a positive 
and negative disk respectively.
Let $[P^S]$ and $[P^L]$ be the contribution of a small 
and large pants respectively.
\begin{theorem}\label{local RRs in 2 dim}
$$
[BS^+]=1,\quad
[BS^-]=-1,\quad
[D^+]=1,\quad
[D^-]=0,\quad
[P^S]=0,\quad
[P^L]=-1
$$
\end{theorem}
\begin{proof}
These are consequences of
Lemma~{\ref{basic relations}} and
Lemma~\ref{Relations from RR theorem}
below.
\end{proof}

\begin{lemma}\label{basic relations}
$[P^L]+[BS^+]=[P^S],\quad
[BS^-]+[BS^+]=0,\quad
[D^-]+[BS^+]=[D^+].$
\end{lemma}
\begin{proof}
The three relations are shown in a similar way.
We just show the first relation.
Let $\epsilon$ be a small positive number.
Let  $P^S$  and $P^L$ be 
the small pants and the large pants
in Example~\ref{Surface with a pants 
decomposition with flat line bundle}.
%
%
Let $A$ be an oriented annulus with $U(1)$-connection
such that the connection is flat near the boundary.
Suppose $A$ is positive-BS type and both the holonomies
of the two boundary components are $e^{\sqrt{-1} 2 \epsilon}$
for the orientation as boundary of $A$.
Patch $P^L$ and $A$ together along the boundary components
with holonomies $e^{- \sqrt{-1} 2\epsilon}$ and
$e^{\sqrt{-1} 2 \epsilon}$ to obtain an another pair of pants 
with a $U(1)$-connection. The glued $U(1)$-connection
can be deformed to a flat $U(1)$-connection isomorphic to
the one on $P^S$ without changing the connection near boundary
components.
\end{proof}

\begin{lemma}\label{Relations from RR theorem}
$[BS^+]=1,\quad
[P^S]+[P^L] =-1, \quad
[D^+]+[D^-] =1.$
\end{lemma}
\begin{proof}
The three relations are consequences of
Example~\ref{A torus over a circle with degree $n$ line bundle},
Example~\ref{Sphere with trivial line bundle}, and
Example~\ref{Surface with a pants decomposition with flat line bundle}
respectively.
\end{proof}

\begin{remark}
It is possible to show $[BS^+]=1$ directly without appealing the Riemann-Roch theorem in the following way. We put $M:=\R\times S^1$ and consider the Hermitian structure $(g, J)$ on it, which is defined by 
\[
\begin{split}
&g(a_1\partial_x+b_1\partial_\theta,a_2\partial_x+b_2\partial_\theta)=a_1a_2+b_1b_2,\\
&J(\partial_x)=\partial_\theta , J(\partial_\theta)=-\partial_x. 
\end{split}
\]

Let $T_{\fiber}M$ be the tangent bundle along fibers of the first projection. Then, as complex vector bundles, $T_{\fiber}M\otimes_{\R} \C$ is identified with $(TM,J)$ by 
\begin{equation}\label{identification}
T_{\fiber}M\otimes_{\R} \C \to (TM,J),\ \partial_\theta \otimes_{\R} (x+\sqrt{-1}y)\mapsto \ x\partial_\theta +yJ\partial_\theta . 
\end{equation}

Let $L=M\times \C$ be the trivial complex line bundle on $M$. For $0<\epsilon< 1$ let $\rho(x)$ be a smooth increasing function on $\R$ with $\rho(0)=0$, $\rho(x)\equiv \epsilon$ for sufficiently large $x$ and $\rho(x)\equiv -\epsilon$ for sufficiently small $x$. Consider the $U(1)$-connection on $L$ of the form $\nabla =d-\sqrt{-1}\rho(x)d\theta$. 

Let $W:=\wedge^{\bullet}(TM,J)\otimes_{\C} L$ be the $\Z/2$-graded $Cl(TM)$-module bundle over $M$ whose $Cl(TM)$-module structure is defined by \cite[pp.38, (5.25)]{LawsonMichelsohn}. We take a Dirac-type operator $D$ acting on $\Gamma(W)$ to be the Dolbeault operator. Under the identification~\eqref{identification} we also take $D_{\fiber}$ to be the family of de Rham operators along fibers. It is easy to check that $D_{\fiber}$ satisfies the fourth property in Definition~4.3. Then the deformed operator $D_t=D+tD_{\fiber}$ is written in the following way
\begin{align}
D_ts=&\partial_x\otimes \left( \partial_xs_0+\sqrt{-1}(1+t)\partial_\theta s_0+(1+t)\rho(x)s_0\right) \nonumber \\
&-\left( \partial_xs_1-\sqrt{-1}(1+t)\partial_\theta s_1-(1+t)\rho(x)s_1\right) \nonumber
\end{align}
for $s=s_0+\partial_x\otimes s_1\in \Gamma (W)$, where $s_0\in \Gamma (L)$ and $\partial_x\otimes s_1\in \Gamma ((TM,J)\otimes_{\C}L)$ are even and odd parts of $s$, respectively. 

For an $L^2$-section $s_0$ of $L$, we first solve the equation 
\begin{equation}\label{plus}
0=\partial_xs_0+\sqrt{-1}(1+t)\partial_\theta s_0+(1+t)\rho(x)s_0. 
\end{equation}
By taking the Fourier expansion of $s_0$ with respect to $\theta$, $s_0$ is written as 
\begin{equation*}\label{Fourier}
s_0=\sum_{n\in \Z}a_n(x)e^{\sqrt{-1}n\theta}.  
\end{equation*}
Then, $s_0$ satisfies \eqref{plus} if and only if each $a_n$ is of the form 
\[
a_n(x)=c_n\exp \left( (1+t)\int_0^xn-\rho(x)dx\right)
\]
for some constant $c_n$. Since $\rho(x)\equiv \pm \epsilon$ for sufficiently large, or small $x$ and since $s$ is a $L^2$-section, it is easy to see that $c_n=0$ except for $n=0$. This implies that the kernel of the even part of $D_t$ is one-dimensional. 

Next we solve the equation
\begin{equation*}\label{minus}
0=\partial_xs_1-\sqrt{-1}(1+t)\partial_\theta s_1-(1+t)\rho(x)s_1. 
\end{equation*}
By the similar argument we can show that $c_n=0$ for all $n\in \Z$. This implies that the kernel of the odd part of $D_t$ is zero-dimensional. 
Thus, we have $[BS^+]=1$. 

By similar arguments we can also show that $[BS^-]=-1$, $[D^+]=1$, and $[D^-]=0$. 
\end{remark}


\subsection{Higher dimensional Bohr-Sommerfeld fibers}
We show the following.   
 
\begin{theorem}\label{local RR of BS}
In symplectic formulation, the local Riemann-Roch number
of a non-singular connected Bohr-Sommerfeld fiber is one.
\end{theorem}
\begin{proof}
It is known that the neighborhoods of two Bohr-Sommerfeld fibers are
isomorphic to each other together with prequantizing line bundle
with connection: Recall that 
the fibers in a neighborhood of a Bohr-Sommerfeld fiber
are parameterized by their  periods. If we fix a local Lagrangian section,
and a trivialization of the first homology group of the local fibers,
then we can write down a canonical coordinate.
Therefore it suffices to give one example for which
the claim is satisfied.
An example of a Lagrangian fibration with 
exactly one $n$-dimensional Bohr-Sommerfeld fiber is given
by the product of $n$-copies of the fiber bundle structure of
the torus $T^2_N$ for $N=1$
in 
Example~\ref{A torus over a circle with degree $n$ line bundle}.
In this case 
our convention of the orientation for the symplectic manifold 
coincides with the product orientation. The Riemann-Roch number
is equal to one because it is equal to 
the $n$-th power of $RR(T^2_1)=1$. 
\end{proof}

As a corollary we have the following, which is already 
shown by J.~E.~Andersen by using index theorem.

\begin{corollary}[\cite{A}]
For a Lagrangian fibration without singular fibers
over a closed symplectic manifold with a prequantizing line bundle, 
the number of Bohr-Sommerfeld fibers is equal to the Riemann-Roch number. 
\end{corollary}
\begin{remark}
It would be expected that,
if we use appropriate boundary condition, then
it would be possible to define a local Riemann-Roch number
for the product $D^+ \times P^L$, and moreover
it would be equal to the product  $[D^+][P^L]$, i.e., $-1$.
A crucial problem here is that 
there is no Lagrangian fibration structure
on the whole neighborhood of the boundary
of $D^+ \times P^L$. 
In fact it is possible to extend our formulation to
such cases. We will discuss this elsewhere \cite{FFY}.
\end{remark}

\section{Comments}
It is possible to extend our construction
for various situations.
\begin{enumerate}
\item
Isotropic fibrations:
When we have a integrable system which is
not necessary complete,
if all the orbits are \lq\lq periodic" and form tori,
we can  extend our argument.
It would be an interesting problem
to investigate the case when the orbits are not
periodic.
\item
Manifolds with boundaries and corners:
Our definition of the local index
is related to manifolds with boundaries.
For manifolds with coners, it is possible
to extend our construction.
\item
Equivariant and family version:
Our constructin is natural, so if a compact
Lie group acts and preserves the data,
then everything is formulated equivariantly.
Similarly we have a family version of our
construction.
\item
Equivariant mod-2 indices:
A modification of the localization property
explained in this paper can be applied to
define {\it $G$-equivariant mod-2 indices}
valued in $R(G)/RO(G)$ or $R(G)/RSp(G)$
for even dimensional 
$G$-spin$^{c}$-manifolds with 
$G$-spin structures on its end
\cite{FK}.
\end{enumerate}
We will discuss them elsewhere \cite{FFY,FFY3}.

\appendix

\section{Spin${}^c$-structures}\label{spin^c-structures}

In this appendix we recall our convention for 
spin$^{c}$-strctures on oriented manifolds. 
(See \cite[pp.54]{index 1} and \cite[pp.131-132]{LawsonMichelsohn}.)
A spin$^c$ structure is usually defined for 
an oriented Riemannian manifold. 
In this paper we take a convention of 
spin$^c$ structures which 
do not need any Riemannian metrics. 
In fact a spin$^c$ structure itself is defined 
at the principal bundle level as follows. 
Let $GL^{+}_{m}({\R})$ be the group of orientation preserving linear automorphisms of $\R^m$. Since $GL^{+}_{m}({\R})$ has the same homotopy type 
as that of $SO_{m}(\R)$, there is a unique non-trivial 
double covering of $GL^{+}_{m}({\R})$ when $m>1$.
We denote it by 
$p:\tilde{GL^{+}_{m}}({\R})\to GL^{+}_{m}({\R})$. 
When $m=1$, we define 
$\tilde{GL^{+}_{1}}({\R}):=GL^{+}_{1}({\R}) \times (\Z/2\Z)$
and $p$ to be the projection onto the first factor.

Consider the diagonal action of $\Z/2\Z$ on 
$\widetilde{GL^{+}_{m}}({\R})\times\C^{\times}$, where 
the action on the first factor is the deck transformation 
and the action on the second factor is defined by $z\mapsto -z$. 
Let $\widetilde{GL^{+}_{m}}({\R})\times_{\Z/2\Z}\C^{\times}$ be 
the quotient group by this diagonal action. 
Note that there is a canonical homomorphism 
$\hat p:\widetilde{GL^{+}_{m}}({\R})\times_{\Z/2\Z}\C^{\times} 
\to GL^{+}_{m}({\R})$ 
defined by 
$$
\hat p:[\tilde g,z]\mapsto p(\tilde g). 
$$

\begin{definition}
Let $M$ be an oriented manifold and 
$P_M$ the associated frame bundle over $M$, which is a 
principal $GL^{+}_{m}({\R})$-bundle. 
A {\it spin}$^c$-{\it structure} on $M$ is a 
pair $(\tilde P_M,q_M)$ satisfying the following two conditions. 
\begin{enumerate}
\item $\tilde P_M$ is a principal 
$\widetilde{GL^{+}_{m}}({\R})\times_{\Z/2\Z}\C^{\times}$-bundle
over $M$. 
\item $q_M$ is a bundle map from $\tilde P_M$ to $P_M$ 
which is equivariant with respect to the canonical 
homomorphism $\hat p$. 
\end{enumerate}
\end{definition}

\begin{remark}
Though a spin$^c$ structure can be defined 
without any Riemannian metric, 
we need to fix a Riemannian metric to define a Clifford module bundle 
over an oriented manifold. 
\end{remark}

\begin{remark}
The natural embedding $GL_n(\C)\hookrightarrow GL_{2n}^+(\R)$ 
induces a homomorphism 
$GL_n(\C)\to \tilde{GL^{+}_{2n}}({\R})\times_{\Z/2}\C^{\times}$. 
This means that an almost complex manifold has a 
canonical spin$^c$ structure in our convention. \
\end{remark}

\section*{Acknowledgements}
The author would like to thank 
Kiyonori Gomi for stimulating discussions.

\par\noindent{\scshape 
\small 
Hajime FUJITA
\\ Department of Mathematics,
Gakushuin University, \\ 1-5-1 Mejiro, Toshima-ku, Tokyo
171-8588, Japan.}
\par\noindent{\ttfamily hajime@math.gakushuin.ac.jp}

\vspace*{2mm}
\par\noindent{\scshape 
\small
Mikio FURUTA
\\ Graduate School of Mathematical Sciences,
The University of Tokyo, \\3-8-1 Komaba,
Meguro-ku, Tokyo 153-9814, Japan.}
\par\noindent{\ttfamily furuta@ms.u-tokyo.ac.jp}

\vspace*{2mm}
\par\noindent{\scshape 
\small 
Takahiko YOSHIDA
\\ Department of Mathematics, Graduate School of Science and Technology, Meiji University, \\1-1-1 Higashimita, 
Tama-ku, Kawasaki 214-8571, Japan.}
\par\noindent{\ttfamily takahiko@math.meiji.ac.jp}
\end{document}

%% file: BS.tex
\begin{picture}(0,0)%
\includegraphics{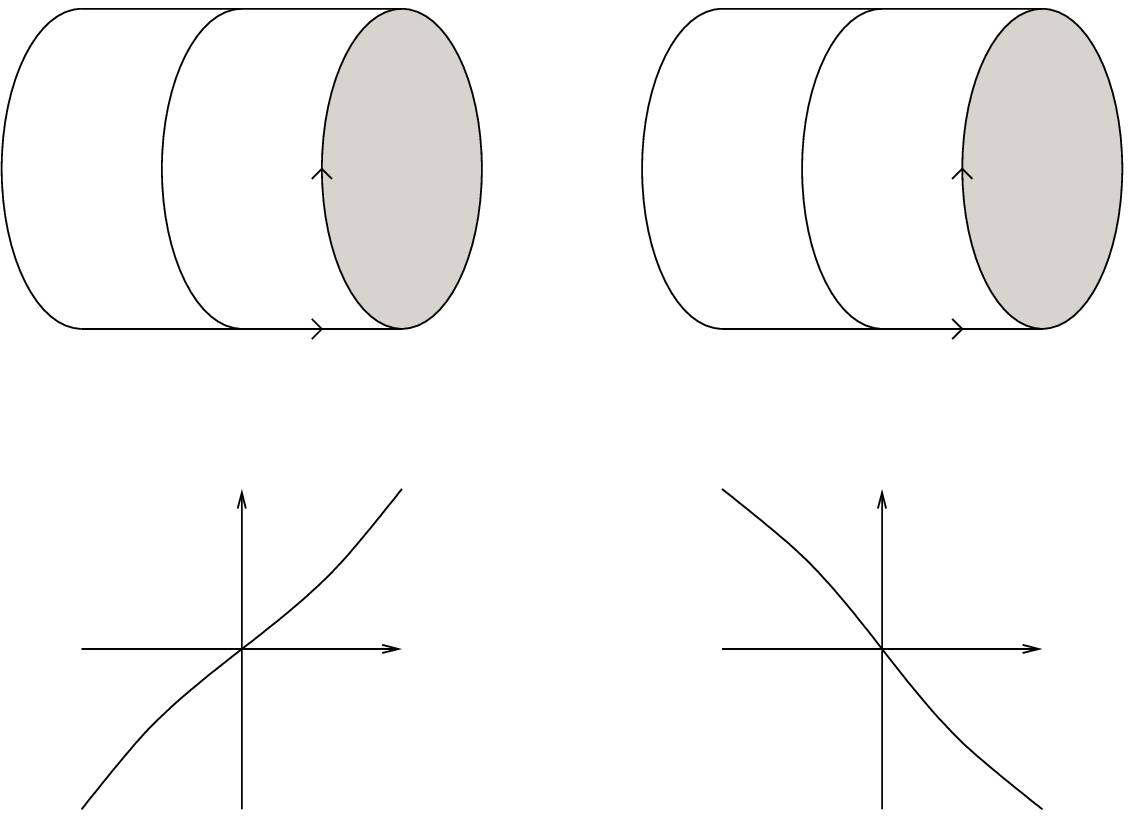}%
\end{picture}%
\setlength{\unitlength}{3947sp}%
\begingroup\makeatletter\ifx\SetFigFont\undefined%
\gdef\SetFigFont#1#2#3#4#5{%
  \reset@font\fontsize{#1}{#2pt}%
  \fontfamily{#3}\fontseries{#4}\fontshape{#5}%
  \selectfont}%
\fi\endgroup%
\begin{picture}(5395,4228)(341,-3725)
\put(2326,-2611){\makebox(0,0)[lb]{\smash{{\SetFigFont{12}{14.4}{\familydefault}{\mddefault}{\updefault}{\color[rgb]{0,0,0}$x$}%
}}}}
\put(1426,-1786){\makebox(0,0)[lb]{\smash{{\SetFigFont{12}{14.4}{\familydefault}{\mddefault}{\updefault}{\color[rgb]{0,0,0}$h(x)$}%
}}}}
\put(5401,-2611){\makebox(0,0)[lb]{\smash{{\SetFigFont{12}{14.4}{\familydefault}{\mddefault}{\updefault}{\color[rgb]{0,0,0}$x$}%
}}}}
\put(4501,-1786){\makebox(0,0)[lb]{\smash{{\SetFigFont{12}{14.4}{\familydefault}{\mddefault}{\updefault}{\color[rgb]{0,0,0}$h(x)$}%
}}}}
\put(1351,-3661){\makebox(0,0)[lb]{\smash{{\SetFigFont{12}{14.4}{\familydefault}{\mddefault}{\updefault}{\color[rgb]{0,0,0}$BS^+$}%
}}}}
\put(4426,-3661){\makebox(0,0)[lb]{\smash{{\SetFigFont{12}{14.4}{\familydefault}{\mddefault}{\updefault}{\color[rgb]{0,0,0}$BS^-$}%
}}}}
\put(1801,-1336){\makebox(0,0)[lb]{\smash{{\SetFigFont{12}{14.4}{\familydefault}{\mddefault}{\updefault}{\color[rgb]{0,0,0}$x$}%
}}}}
\put(4876,-1336){\makebox(0,0)[lb]{\smash{{\SetFigFont{12}{14.4}{\familydefault}{\mddefault}{\updefault}{\color[rgb]{0,0,0}$x$}%
}}}}
\put(1576,-361){\makebox(0,0)[lb]{\smash{{\SetFigFont{12}{14.4}{\familydefault}{\mddefault}{\updefault}{\color[rgb]{0,0,0}$\theta$}%
}}}}
\put(4651,-361){\makebox(0,0)[lb]{\smash{{\SetFigFont{12}{14.4}{\familydefault}{\mddefault}{\updefault}{\color[rgb]{0,0,0}$\theta$}%
}}}}
\end{picture}%

%% file: Disc.tex
\begin{picture}(0,0)%
\epsfig{file=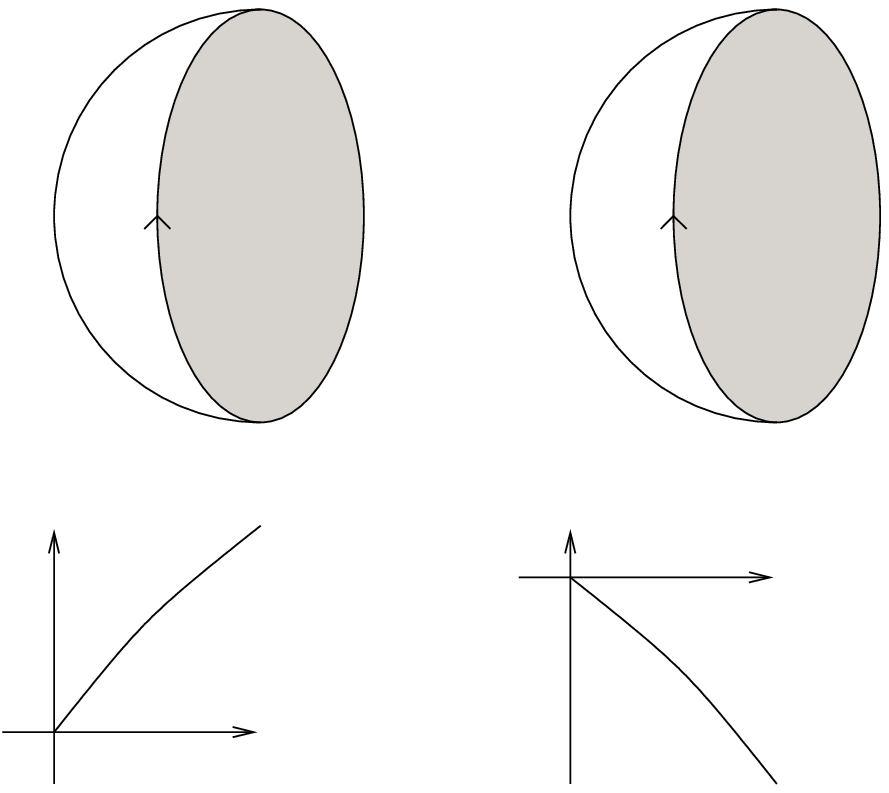}%
\end{picture}%
\setlength{\unitlength}{3947sp}%
\begingroup\makeatletter\ifx\SetFigFont\undefined%
\gdef\SetFigFont#1#2#3#4#5{%
  \reset@font\fontsize{#1}{#2pt}%
  \fontfamily{#3}\fontseries{#4}\fontshape{#5}%
  \selectfont}%
\fi\endgroup%
\begin{picture}(4233,4077)(289,-3591)
\put(728,-530){\makebox(0,0)[lb]{\smash{{\SetFigFont{12}{14.4}{\familydefault}{\mddefault}{\updefault}{\color[rgb]{0,0,0}$\theta$}%
}}}}
\put(1583,-3023){\makebox(0,0)[lb]{\smash{{\SetFigFont{12}{14.4}{\familydefault}{\mddefault}{\updefault}{\color[rgb]{0,0,0}$x$}%
}}}}
\put(443,-1954){\makebox(0,0)[lb]{\smash{{\SetFigFont{12}{14.4}{\familydefault}{\mddefault}{\updefault}{\color[rgb]{0,0,0}$h(r)$}%
}}}}
\put(2866,-1883){\makebox(0,0)[lb]{\smash{{\SetFigFont{12}{14.4}{\familydefault}{\mddefault}{\updefault}{\color[rgb]{0,0,0}$h(r)$}%
}}}}
\put(3578,-3522){\makebox(0,0)[lb]{\smash{{\SetFigFont{12}{14.4}{\familydefault}{\mddefault}{\updefault}{\color[rgb]{0,0,0}$D^-$}%
}}}}
\put(4077,-2311){\makebox(0,0)[lb]{\smash{{\SetFigFont{12}{14.4}{\familydefault}{\mddefault}{\updefault}{\color[rgb]{0,0,0}$x$}%
}}}}
\put(3222,-530){\makebox(0,0)[lb]{\smash{{\SetFigFont{12}{14.4}{\familydefault}{\mddefault}{\updefault}{\color[rgb]{0,0,0}$\theta$}%
}}}}
\put(1085,-3522){\makebox(0,0)[lb]{\smash{{\SetFigFont{12}{14.4}{\familydefault}{\mddefault}{\updefault}{\color[rgb]{0,0,0}$D^+$}%
}}}}
\end{picture}%

%% file: bsnew.bbl
\begin{thebibliography}{11} 

\bibitem{A}
J. E. Andersen, {\it Geometric quantization of symplectic manifolds with respect to reducible non-negative polarizations}, Commun. Math. Phys. 183 (1997), no. 2, 401-421.

\bibitem{BPU}
D. Borthwick, T. Paul and A. Uribe,
{\it 
Legendrian distributions with applications 
to relative Poincar\`{e} 
series}, Invent. 
Math. 122 (1995), no. 2, 359-402. 
\bibitem{DGMW}
H. Duistermaat, V. Guillemin, E. Meinrenken 
and S. Wu, 
{\it Symplectic reduction 
and Riemann-Roch for circle
actions}, Math. Res. Lett. 2 (1995), no. 3, 259-266.
\bibitem{FFY} H.~Fujita, M.~Furuta and T.~Yoshida, {\it Torus fibrations and localization of index II}, arXiv:0910.0358.
\bibitem{FFY3} H.~Fujita, M.~Furuta and T.~Yoshida, {\it Torus fibrations and localization of index III}, in preparation.
\bibitem{index 1} 
M.~Furuta, Index Theorem 1. (Japanese)
Iwanami Series in Modern 
Mathematics. Iwanami Shoten, Publishers, 
Tokyo, 1999.
(English translation by Kaoru Ono,
Translations of Mathematical Monographs, 235. 
American Mathematical Society, 
Providence, RI, 2007)

\bibitem{index 2} 
M.~Furuta, Index Theorem 2. (Japanese) 
Iwanami Series in Modern 
Mathematics. Iwanami Shoten, Publishers, 
Tokyo, 2002.

\bibitem{FK}
M.~Furuta and Y.~Kametani,
{\it Equivariant mod-2 index}, 
in preparation.

\bibitem{Gomi} K.~Gomi,
{\it Twisted K-theory and finite-dimensional approximation}, to appear in Commun. Math. Phys. Also available at arXiv:0803.2327. 

\bibitem{Hamilton}
M. D. Hamilton, {\it
Locally toric manifolds and singular Bohr-Sommerfeld leaves}, to appear in Mem. Amer. Math. Soc. Also available at arXiv:0709.4058.

\bibitem{Hamilton-Miranda}M. D. Hamilton {and} E. Miranda, 
{\it Geometric quantization of integrable systems with hyperbolic singularities}, to appear in Annales de l'Institut Fourier. Also available at arXiv:0808.0338.

\bibitem{Lerman}
E. Lerman, {\it Symplectic cuts}, Math. 
Res. Lett. 2 (1995), no.3, 247-258.

\bibitem{LawsonMichelsohn}
H.B. Lawson and M.-L. Michelsohn, 
Spin geometry, Princeton Math. Series, 38, 
Princeton University Press, Princeton, NJ, 1989.

\bibitem{TianZhang}
Y. Tian and W. Zhang,
{\it An analytic proof of the geometric quantization 
conjecture of Guillemin-Sternberg},
Invent. Math. 132 (1998), no. 2, 229-259.

\end{thebibliography}
